\documentclass[copyright,creativecommons]{eptcs}
\providecommand{\event}{ACT 2022} 
\usepackage{breakurl}             

\usepackage{amsthm}
\usepackage{amsfonts}

\usepackage{amssymb}
\usepackage{amsmath}
\usepackage{hyperref}
\usepackage{cleveref}
\usepackage[all]{xy}
\usepackage{tikz-cd}
\usepackage{etoolbox}
\apptocmd{\sloppy}{\hbadness 10000\relax}{}{}

\newtheorem{theorem}{Theorem}
\newtheorem*{theorem*}{Theorem}
\newtheorem{proposition}{Proposition}
\newtheorem*{proposition*}{Proposition}
\newtheorem{definition}{Definition}
\newtheorem*{definition*}{Definition}
\newtheorem{remark}{Remark}
\newtheorem*{remark*}{Remark}
\newtheorem{lemma}{Lemma}
\newtheorem*{lemma*}{Lemma}

\newtheorem*{corollary*}{Corollary}

\DeclareRobustCommand{\coprod}{\mathop{\text{\fakecoprod}}}
\newcommand{\fakecoprod}{
  \sbox0{$\prod$}
  \smash{\raisebox{\dimexpr.9625\depth-\dp0}{\scalebox{1}[-1]{$\prod$}}}
  \vphantom{$\prod$}
}

\DeclareMathOperator{\ob}{Ob}

\DeclareMathOperator{\id}{Id}

\newcommand{\Get}{{\mathrm{get}}}
\newcommand{\Put}{{\mathrm{put}}}

\newcommand{\Mod}{{\mathbf{Mod}}}
\newcommand{\Comod}{{\mathbf{Comod}}}
\newcommand{\CAlg}{{\mathbf{CAlg}}}
\newcommand{\CCoalg}{{\mathbf{CCoalg}}}
\newcommand{\Span}{{\mathbf{Span}}}
\newcommand{\Cospan}{{\mathbf{Cospan}}}
\newcommand{\CMon}{{\mathbf{CMon}}}
\newcommand{\CComon}{{\mathbf{CComon}}}
\newcommand{\Set}{\mathbf{Set}}
\newcommand{\B}{\mathcal{B}}
\newcommand{\D}{\mathcal{D}}
\newcommand{\M}{\mathcal{M}}
\newcommand{\Cat}{\mathcal{C}}
\newcommand{\CAT}{\mathbf{Cat}}
\newcommand{\PROF}{\mathbf{Prof}}

\newcommand{\superscript}[1]{^{\textnormal{#1}}}
\newcommand{\anon}{{\text{\textendash}}}
\newcommand{\anonfirst}{{\text{\textendash}}}
\newcommand{\anonsecond}{{\text{=}}}
\newcommand{\angles}[2]{{\left\langle{#1} \, | \, {#2}\right\rangle}}
\newcommand*\circled[1]{\tikz[baseline={([yshift=-0.65ex]current bounding box.center)}]{
   \node[shape=circle,draw,inner sep=1pt] (char) {#1};}}
\newcommand{\actL}{{\circled{\tiny$\mathsf{L}$}}}
\newcommand{\actR}{{\circled{\tiny$\mathsf{R}$}}}
\newcommand{\icatL}{{\mathcal{L}}}
\newcommand{\icatR}{{\mathcal{R}}}
\newcommand{\starL}{^{*}}
\newcommand{\starR}{^{*'}}

\newcommand{\Lens}{{\mathbf{Lens}}}
\newcommand{\DLens}{{\mathbf{DLens}}}
\newcommand{\DPrism}{{\mathbf{DPrism}}}
\newcommand{\Optic}{{\mathbf{Optic}}}
\newcommand{\Tamb}{{\mathbf{Tamb}}}

\title{Dependent Optics}
\author{Pietro Vertechi
\email{pietro.vertechi@protonmail.com}
}
\def\titlerunning{Dependent Optics}
\def\authorrunning{P. Vertechi}

\begin{document}

\maketitle

\begin{abstract}
    A wide variety of bidirectional data accessors, ranging from {\em mixed optics} to {\em functor lenses}, can be formalized within a unique framework---{\em dependent optics}. Starting from two indexed categories, which encode what maps are allowed in the forward and backward directions, we define the category of dependent optics and establish under what assumptions it has coproducts. Different choices of indexed categories correspond to different families of optics: we discuss dependent lenses and prisms, as well as closed dependent optics. We introduce the notion of {\em Tambara representation} and use it to classify contravariant functors from the category of optics, thus generalizing the {\em profunctor encoding} of optics to the dependent case.
\end{abstract}

\section{Introduction}

Lenses~\cite{bohannon2006relational,johnson2012lenses} are composable, bidirectional data accessors. They can be thought of as a collection of two methods: a $\Get$ method, to access a particular field of a data structure, and a $\Put$ method, to build a new instance of the data structure with an updated field value. Lenses and their more recent generalization, optics, have been implemented and explored in the popular {\em Haskell} library $\mathrm{lens}$~\cite{kmett}. The possible fields of application vary widely, from game theory~\cite{ghani2018compositional} to automatic differentiation~\cite{fischer}.

The current formalization of optics~\cite{clarke2020profunctor,riley2018categories} extends the original theory of lenses from Cartesian categories to arbitrary symmetric monoidal categories, or even {\em actegories}, thus including under a unique formalism a wide variety of data accessors. Unfortunately, this approach fails to include a distinct elegant generalization of lenses, namely {\em functor lenses}~\cite{spivak2019generalized}: every pseudofunctor $\icatR\colon \Cat\superscript{op} \rightarrow \CAT$ induces a fibration of categories $\Lens_\icatR \rightarrow \Cat$ via the Grothendieck construction on the pointwise opposite of $\icatR$. Classical categories of lenses, as well as novel examples, can be obtained with this approach.

The aim of this work is to develop a common generalization of the theories of optics and functor lenses, via the theory of bicategories and pseudofunctors. In \cref{sec:dependent_optics} we lay the fundamental definition of {\em dependent optic} and show that it encompasses both regular optics and functor lenses. We show under what conditions the category of dependent optics has coproducts. In \cref{sec:examples} we give some examples of dependent optics---dependent (monoidal) lenses, dependent (monoidal) prisms, and closed dependent optics---where the key ingredient is the operation of tensoring (co)modules over a (co)monoid in a symmetric monoidal category. Finally, in \cref{sec:tambara}, we establish the notion of {\em Tambara representation}. We show that $\D$-valued Tambara representations are equivalent to contravariant functors from the category of optics to an arbitrary category $\D$, thus generalizing the profunctor encoding of optics to the dependent case.

\section{Dependent optics}
\label{sec:dependent_optics}

Classically, a lens from a domain $(X, X')$ to a codomain $(Y, Y')$ has a $\Get$ method $\Get\colon X \rightarrow Y$ and a $\Put$ method $\Put\colon X \times Y' \rightarrow X'$. This definition suggests two broad classes of generalizations. One approach, {\em functor lenses}~\cite{spivak2019generalized}, replaces the domain and codomain with pairs $(X, P)$ and $(Y, Q)$, where the objects $P$ and $Q$ live in categories parameterized by $X$ and $Y$ respectively. The $\Put$ method is then encoded as a map $\Put\colon \Get^*(Q) \rightarrow P$. Another approach, {\em mixed optics}~\cite{clarke2020profunctor,riley2018categories}, replaces the Cartesian product $\times$ with two actions $\actL, \actR$ of a shared monoidal category $\M$ on categories $\Cat_L, \Cat_R$. This requires to tweak the original definition of lens to an equivalent one expressed via a coend
\begin{equation*}
    \int^{M\in\M} \Cat_L(X, M\actL Y) \times \Cat_R(M\actR Y', X').
\end{equation*}
\begin{remark}
    Some authors (see for instance~\cite{clarke2020profunctor}) work in the setting of enriched categories, so that the above coend is not taken in $\Set$ but rather in some monoidal category $\mathcal{V}$. For simplicity, in this article we will work in the standard non-enriched setting.
\end{remark}
It follows from the Yoneda reduction lemma~\cite[Lm.~1.2.2]{riley2018categories} that this definition recovers classical lenses when a Cartesian category acts on itself. From a practical perspective, optics are equivalence classes of pairs of morphisms
\begin{equation*}
    l \colon X \rightarrow M \actL Y
    \quad \text{ and }\quad
    r \colon M \actR Y' \rightarrow X',
\end{equation*}
where $M \in \ob(\M)$ is called the {\em representative}.

The aim of this section is to establish a general definition of {\em dependent optics} which encompasses both previous generalizations of lenses---functor lenses and optics. The definition is entirely analogous to the definition of optics, but the monoidal actions are replaced by $\B$-indexed categories, where $\B$ is a {\em bicategory}~\cite{benabou1967introduction} (see also~\cite{lack20102} for a more modern treatment). We will consider the bicategory as a {\em category weakly enriched in categories}, hence the notation $\B(A, B)$, for $A, B \in \ob(\B)$, will represent the category of morphisms from $A$ to $B$.

To encode the data of a $\B$-indexed category $\icatL$, i.e. a pseudofunctor $\icatL\colon\B\superscript{op}\rightarrow \CAT$, we will use the following notation. $\icatL^{A}$ for $A \in \ob(\B)$ denotes the category $\icatL(A)$ and $f\starL$ for $f \in\ob(\B(A, B))$ denotes the functor $\icatL(f)$. Throughout this manuscript, we will work with two pseudofunctors, $\icatL$ and $\icatR$. To avoid ambiguities, we will use the notation $f\starR$ to denote $\icatR(f)$.

\begin{definition}\label{def:optics}
    Let $\B$ be a bicategory. Let $\icatL, \icatR$ be $\B$-indexed categories. The category $\Optic_{\icatL, \icatR}$ of {\em dependent optics} has, as objects, triplets $(X, X')^{A}$, with $A \in \ob(\B)$, $X \in \ob(\icatL^{A})$, $X' \in \ob(\icatR^{A})$. Morphisms between $(X, X')^{A}$ and $(Y, Y')^{B}$ are given by the following coend:
    \begin{equation}\label{eq:morphisms}
        \Optic_{\icatL, \icatR}\left((X, X')^{A}, (Y, Y')^{B}\right) = \int^{f \in \B(A, B)} {\icatL^{A}}(X, f\starL Y) \times {\icatR^{A}}(f\starR Y', X').
    \end{equation}
    To refer to specific morphisms explicitly, we denote by $\angles{l}{r}$ the morphism given by $l \in {\icatL^{A}}(X, f\starL Y)$ and $r \in {\icatR^{A}}(f\starR Y', X')$, and we say that it has {\em representative} $f$.
\end{definition}

More explicitly, morphisms in $\Optic_{\icatL, \icatR}\left((X, X')^{A}, (Y, Y')^{B}\right)$ are equivalence classes of pairs $(l ,r)$ with $l \colon X \rightarrow f\starL Y$ and $r \colon f\starR Y' \rightarrow X'$, where $f\colon A \rightarrow B$ is called the representative. The equivalence relation is generated by
\begin{equation*}
    \left(\icatL(m)_{Y}\circ l, \, r\right) \sim \left(l, \, r \circ \icatR(m)_{Y'}\right),
\end{equation*}
with $m \colon f \Rightarrow g$, $l \colon X \rightarrow f\starL Y$ and $r \colon g\starR Y' \rightarrow X'$, where $f, g\colon A \rightrightarrows B$ are parallel 1-morphisms in $\B$.

\begin{remark}
    Here and in what follows we assume that the above coend exists, either because the category $\B(A, B)$ is small (and small colimits exist in $\Set$), or because we can compute it explicitly.
\end{remark}

The generalization of optics via $\B$-indexed categories, rather than monoidal actions, has been proposed in~\cite{milewski2022compound}, where composition of optics is explained in terms of Kan extensions. Here, we will adopt a direct, explicit approach. While the chosen formalisms are different, the two definitions have been shown to be equivalent in~\cite[Ex.~4.2]{capucci2022seeing}.

Let $\theta, \theta'$ encode the coherence natural transformations for $\icatL, \icatR$ respectively. In particular, we have natural isomorphisms
\begin{equation*}
    \theta_A\colon\id_{\icatL^{A}}\Rightarrow \id_A\starL
    \quad\text{ and } \quad
    \theta_{f, g}\colon f\starL\circ g\starL\Rightarrow (g \circ f)\starL.
\end{equation*}
$\theta'_A$ and $\theta'_{f, g}$ are defined in an analogous way.
The identity optic is defined as follows:
\begin{equation}\label{eq:identity_morphism}
    \id_{(X, X')^{A}} := \angles{(\theta_A)_X}{(\theta_A^{\prime-1})_{X'}}.
\end{equation}
The map
\begin{equation*}
    \begin{tikzcd}
    {\icatL^{B}}(Y, g\starL Z) \times {\icatR^{B}}(g\starR Z', Y') \times {\icatL^{A}}(X, f\starL Y) \times {\icatR^{A}}(f\starR Y', X')\arrow{d}\\
    \Optic_{\icatL, \icatR}\left((X, X')^{A}, (Z, Z')^{C}\right)
    \end{tikzcd}
\end{equation*}
given by
\begin{equation}\label{eq:morphism_composition}
    \angles{l_2}{r_2}\circ\angles{l_1}{r_1} = \angles{(\theta_{f, g})_{Z}\circ f\starL(l_2)\circ l_1}{r_1\circ f\starR(r_2)\circ (\theta^{\prime-1}_{f,g})_{Z'}}
\end{equation}
is extranatural in $f, g$ and thus induces a composition function
\begin{equation*}
    \begin{tikzcd}
        \Optic_{\icatL, \icatR}\left((Y, Y')^{B}, (Z, Z')^{C}\right) \times \Optic_{\icatL, \icatR}\left((X, X')^{A}, (Y, Y')^{B}\right) \arrow{d} \\
        \Optic_{\icatL, \icatR}\left((X, X')^{A}, (Z, Z')^{C}\right).
    \end{tikzcd}
\end{equation*}

\begin{theorem}\label{thm:optic_category}
$\Optic_{\icatL, \icatR}$ is a category.
\end{theorem}
\begin{proof}
Checking the category axioms (unity and associativity) is tedious but straightforward. Here, we denote $\lambda_f, \rho_f$ the left and right unitors in $\B$. We use the fact that $\theta_A$ (resp. $\theta^{\prime-1}_A$) is a natural transformation $\id_{\icatL^A} \Rightarrow \id_A\starL$ (resp. $\id_A\starR \Rightarrow \id_{\icatR^A}$), as well as the identity coherence law for a pseudofunctor, keeping in mind that $\CAT$ is a strict 2-category and hence has a trivial unitor. Given an optic $\angles{l}{r}\colon (X, X')^A \rightarrow (Y, Y')^B$ with representative $f$:
\begin{align*}
    \angles{l}{r}\circ\id_{(X, X')^{A}}
    &=\angles{(\theta_{\id_A, f})_{Y}\circ \id_A\starL(l)\circ (\theta_A)_{X}}{(\theta^{\prime-1}_{A})_{X'}\circ \id_A\starR(r)\circ (\theta^{\prime-1}_{\id_A, f})_{Y'}}\\
    &=\angles{(\theta_{\id_A, f})_{Y}\circ (\theta_A)_{f\starL Y} \circ l}{r \circ (\theta^{\prime-1}_A)_{f\starR Y'} \circ (\theta^{\prime-1}_{\id_A, f})_{Y'}}\\
    &=\angles{(\theta_{\id_A, f})_{Y}\circ (\theta_A)_{f\starL Y} \circ l}{r \circ \icatR(\rho_f)_{Y'}}\\
    &=\angles{\icatL(\rho_f)_{Y}\circ (\theta_{\id_A, f})_{Y}\circ (\theta_A)_{f\starL Y} \circ l}{r}\\
    &=\angles{l}{r},
\end{align*}
where $\icatR(\rho_f)_{Y'}$ can be moved to the left as $\icatL(\rho_f)_{Y}$ thanks to the equivalence relation introduced by the coend. Analogously,
\begin{align*}
    \id_{(Y, Y')^{B}}\circ\angles{l}{r}
    &=\angles{(\theta_{f, \id_B})_{Y}\circ f\starL((\theta_{B})_{Y})\circ l}{r\circ f\starR((\theta^{\prime-1}_{B})_{Y'})\circ (\theta^{\prime-1}_{f, \id_B})_{Y'}}\\
    &=\angles{(\theta_{f, \id_B})_{Y}\circ f\starL((\theta_{B})_{Y})\circ l}{r \circ \icatR(\lambda_f)_{Y'}}\\
    &=\angles{\icatL(\lambda_f)_{Y}\circ (\theta_{f, \id_B})_{Y}\circ f\starL((\theta_{B})_{Y})\circ l}{r}\\
    &=\angles{l}{r}.
\end{align*}
To prove associativity, let us consider a sequence of morphisms
\begin{equation*}
    (X, X')^A \xrightarrow{\angles{l_1}{r_1}} (Y, Y')^B \xrightarrow{\angles{l_2}{r_2}} (Z, Z')^C \xrightarrow{\angles{l_3}{r_3}} (W, W')^D,
\end{equation*}
with choices of representatives $f$, $g$, $h$ respectively. Then,
\begin{align*}
    (\angles{l_3}{r_3}\circ\angles{l_2}{r_2})\circ\angles{l_1}{r_1}
    &=\angles{(\theta_{g, h})_{W}\circ g\starL(l_3)\circ l_2}{r_2\circ g\starR(r_3)\circ (\theta^{\prime-1}_{g, h})_{W'}}\circ\angles{l_1}{r_1}\\
    &=\langle (\theta_{f, g; h})_{W}\circ f\starL((\theta_{g, h})_{W})\circ f\starL(g\starL(l_3))\circ f\starL l_2 \circ l_1 \, | \\
    &\quad r_1 \circ f\starR(r_2)\circ f\starR(g\starR(r_3))\circ f\starR((\theta^{\prime-1}_{g, h})_{W'})\circ (\theta^{\prime-1}_{f, g;h})_{W'}\rangle.
\end{align*}
Whereas, when associating in a different order, one has
\begin{align*}
    \angles{l_3}{r_3}\circ(\angles{l_2}{r_2}\circ\angles{l_1}{r_1})
    &=\angles{l_3}{r_3}\circ\angles{(\theta_{f,g})_{Z}\circ f\starL(l_2)\circ l_1}{r_1\circ f\starR(r_2)\circ (\theta^{\prime-1}_{f,g})_{Z'}}\\
    &=\langle (\theta_{f; g, h})_{W}\circ (f; g)\starL(l_3) \circ (\theta_{f,g})_{Z}\circ f\starL(l_2)\circ l_1 \, |\\
    &\quad r_1\circ f\starR(r_2)\circ (\theta^{\prime-1}_{f,g})_{Z'} \circ (f; g)\starR(r_3) \circ (\theta^{\prime-1}_{f; g, h})_{W'}\rangle\\
    &=\langle (\theta_{f; g, h})_{W}\circ (\theta_{f,g})_{h\starL W} \circ f\starL(g\starL(l_3)) \circ f\starL(l_2)\circ l_1 \, |\\
    &\quad r_1\circ f\starR(r_2)\circ f\starR(g\starR(r_3)) \circ (\theta^{\prime-1}_{f,g})_{h\starR W'} \circ (\theta^{\prime-1}_{f; g, h})_{W'}\rangle.
\end{align*}
The two optics are equal, thanks to the relationships
\begin{align*}
    \icatL(\alpha_{f, g, h})_{W}\circ (\theta_{f; g, h})_{W}\circ (\theta_{f,g})_{h\starL W} &= (\theta_{f, g; h})_{W}\circ f\starL((\theta_{g, h})_{W}),\\
    (\theta^{\prime-1}_{f,g})_{h\starR W'} \circ (\theta^{\prime-1}_{f; g, h})_{W'} \circ \icatR(\alpha_{f, g, h}^{-1})_{W'}  &= f\starR((\theta^{\prime-1}_{g, h})_{W'}) \circ (\theta^{\prime-1}_{f, g; h})_{W'},
\end{align*}
where $\alpha_{f, g, h}$ is the associator of $\B$.
\end{proof}

\begin{remark}\label{rm:bicategory}
Unlike the dependent lenses case, here we generally do not have a pseudofunctor $\Optic_{\icatL, \icatR} \rightarrow \B$. Such a pseudofunctor would be ill-defined on morphisms in $\Optic_{\icatL, \icatR}$ due to the equivalence relation imposed by the coend. To obviate this issue, a possible approach worthy of future exploration would be to define a {\em bicategory} of dependent optics where, instead of identifying equivalent 1-morphisms, we add 2-morphisms between them. See~\cite{braithwaite2021fibre} for a purely bicategorical approach to dependent optics. 
\end{remark}

\subsection{Comparison with existent constructions}

Dependent optics simultaneously generalize both mixed optics~\cite{riley2018categories} and functor lenses~\cite{spivak2019generalized}. Intuitively, mixed optics are dependent optics where the bicategory $\B$ has a unique object, whereas functor lenses are dependent optics where $\B$ is a 1-category and the $\B$-indexed category $\icatL$ is trivial.

\begin{proposition}\label{prop:dependent_optics_encompass_optics}
    Mixed optics~\cite[Def.~6.1.1]{riley2018categories} are a particular case of dependent optics, where the source bicategory has a unique object.
\end{proposition}
\begin{proof}
    Let us consider two categories $\Cat_L$ and $\Cat_R$ acted on by a monoidal category $\M$. We can consider the bicategory $\mathbb{B}\M$ obtained by delooping. Explicitly, $\mathbb{B}\M$ has a unique object $*$ with endomorphism category $\mathbb{B}\M(*, *) = \M$, where composition is given by the monoidal structure of $\M$. Then the action of $\M$ on another category $\Cat$ induces a pseudofunctor $\M\rightarrow\CAT$. Under this correspondence, optics for the actions $\psi_L\colon\M \rightarrow [\Cat_L, \Cat_L]$ and $\psi_R\colon\M \rightarrow [\Cat_R, \Cat_R]$ are the same as optics for the corresponding pseudofunctors $\icatL, \icatR\colon \mathbb{B}\M\rightrightarrows \CAT$, hence they are a special case of dependent optics with $\B = (\mathbb{B}\M)\superscript{op}$.
\end{proof}

\begin{proposition}\label{prop:dependent_optics_encompass_functor_lenses}
    Functor lenses, as defined in~\cite{spivak2019generalized}, are a particular case of dependent optics, where the source bicategory $\B$ is a category and the $\B$-indexed category $\icatL$ is trivial.
\end{proposition}
\begin{proof}
    Let $\B$ be a 1-category. In~\cite{spivak2019generalized}, functor lenses are defined as the Grothendieck construction of the pointwise opposite of a $\B$-indexed category $\icatR$. Let $\bullet$ be the terminal $\B$-indexed category. Then,
    \begin{equation*}
        \Lens_\icatR \simeq \Optic_{\bullet, \icatR}.
    \end{equation*}
    Indeed, objects in $\Optic_{\bullet, \icatR}$ are simply pairs $(A, X')$, with $X' \in \icatR^{A}$, as there always is a unique object in $\bullet^{A}$. As $\B$ has no non-trivial 2-morphisms, we have
    \begin{equation*}
        \int^{f \in \B(A, B)} {\icatR^{A}}(f\starR Y', X') \simeq \coprod_{f \in \B(A, B)} {\icatR^{A}}(f\starR Y', X').
    \end{equation*}
\end{proof}

\subsection{Coproducts}

One fundamental motivation for dependent lenses and, more generally, dependent optics is the lack of coproducts in categories of ordinary lenses or optics. This situation is much improved in the dependent case: for instance, adding coproducts to the category of lenses leads naturally to dependent lenses~\cite{braithwaite2021fibre}. In the following proposition, we show that there are general conditions to ensure that the category $\Optic_{\icatL, \icatR}$ has coproducts.

\begin{proposition}\label{prop:coproducts_optics}
    Let $\B$ be a bicategory with finite coproducts. Let us assume that $\icatL, \icatR$ turn finite coproducts in $\B$ into finite products in $\CAT$. Then, $\Optic_{\icatL, \icatR}$ has finite coproducts.
\end{proposition}
\begin{proof}
    Let $\left((X_i, X_i')^{A_i}\right)_{i \in I}$ be a finite family of objects in $\Optic_{\icatL, \icatR}$. Let $A = \coprod_{i \in I} A_i$. Let $\iota_i \colon A_i \hookrightarrow A$ be the inclusions. Let $X \in \ob(\icatL^A)$ and $X' \in \ob(\icatR^A)$ be such that, for all $i \in I$,
    \begin{equation*}
        \iota_i\starL(X) \simeq X_i
        \quad \text{ and } \quad
        \iota_i\starR(X') \simeq X_i'.
    \end{equation*}
    For all $(Y, Y')^{B} \in\ob\left(\Optic_{\icatL, \icatR}\right)$ the following holds:
    \begin{align*}
        \Optic_{\icatL, \icatR}\left((X, X')^{A}, (Y, Y')^{B}\right)
        &=\int^{f\in\B(A, B)} \icatL^{A}(X, f\starL Y) \times \icatR^{A}(f\starR Y', X')\\
        &\simeq\int^{(f_i)_{i\in I}\in\prod_{i \in I}\B(A_i, B)} \prod_{i\in I}\icatL^{A_i}(X_i, f_i\starL Y) \times \icatR^{A_i}(f_i\starR Y', X_i')\\
        &\simeq\prod_{i \in I}\Optic_{\icatL, \icatR}\left((X_i, X_i')^{A_i}, (Y, Y')^{B}\right),
    \end{align*}
    where the last isomorphism is Fubini's theorem for coends, hence $(X, X')^{A}$ is the coproduct of $\left((X_i, X_i')^{A_i}\right)_{i \in I}$.
\end{proof}

\section{Examples}
\label{sec:examples}

Different choices of bicategories and functors give rise to different types of optics, see~\cite{clarke2020profunctor} for an overview of the monoidal case, i.e., $\B = (\mathbb{B}\M)\superscript{op}$, as in \cref{prop:dependent_optics_encompass_optics}. Here, we discuss {\em dependent lenses}~\cite{spivak2019generalized}, {\em dependent prisms}, and generalizations thereof. We then show how the existence of a right adjoint to a given functor can be used to construct further classes of examples of dependent optics. More examples of dependent optics, such as polynomial optics, are described in~\cite{milewski2022compound}.

\subsection{Dependent lenses}
\label{sec:dependent_lenses}

\begin{definition}\label{def:lenses}
    Let $\Cat$ be a finitely complete category. Let $\Span_{\Cat}$ be its bicategory of spans. Let $\Cat / \anon$ be the $\Span_{\Cat}$-indexed category of slices. More explicitly,
    \begin{equation*}
        \Cat / \anon \colon \Span_{\Cat}\superscript{op} \rightarrow \CAT
    \end{equation*}
    is a pseudofunctor that associates to each object $A\in \ob(\Cat)$ the slice category $\Cat / A$. Functoriality is given by pulling back and then pushing forward along the legs of the span.
    We define the category of {\em dependent lenses} as follows:
    \begin{equation*}
        \DLens_{\Cat} := \Optic_{\Cat / \anon, \Cat / \anon}.
    \end{equation*}
\end{definition}

Objects in $\DLens_{\Cat}$ are cospans $X \rightarrow A \leftarrow X'$. Morphisms between two cospans $X \rightarrow A \leftarrow X'$ and $Y \rightarrow B \leftarrow Y'$ are given by
\begin{equation}\label{eq:morphisms_lenses}
    \int^{M \in \Cat / (A\times B)} {\Cat / A}(X, M \times_B Y) \times {\Cat / A}(M \times_B Y', X').
\end{equation}
\Cref{eq:morphisms_lenses} can be visualized as follows. A class of homomorphisms in $\DLens_{\Cat}$ with representative $A \leftarrow M \rightarrow B$ is given by a pair of dotted arrows that make the following diagram commute.
\begin{equation*}
    \begin{tikzcd}
        X \arrow[dotted]{d} \arrow{r} & A & X' \arrow{l}\\
        M \times_B Y \arrow{r} & M \arrow{u} & M \times_B Y' \arrow{l} \arrow[dotted]{u}
    \end{tikzcd}
\end{equation*}
Morphisms in $\DLens_{\Cat}$ can be computed explicitly:
\begin{equation}\label{eq:dependent_lens_derivation}
    \begin{aligned}
        \DLens_{\Cat}\left((X, X')^{A}, (Y, Y')^{B}\right)
        &= \int^{M \in \Cat / (A\times B)} {\Cat / A}(X, M \times_B Y) \times {\Cat / A}(M \times_B Y', X') \\
        &\simeq \int^{M \in \Cat / (A\times B)}\coprod_{X \rightarrow Y}
        {\Cat / (A\times B)}(X, M) \times {\Cat / A}(M \times_B Y', X') \\
        &\simeq \coprod_{X \rightarrow Y}
        \int^{M \in \Cat / (A\times B)}
        {\Cat / (A\times B)}(X, M) \times {\Cat / A}(M \times_B Y', X') \\
        &\simeq \coprod_{X \rightarrow Y} {\Cat / A}(X \times_B Y', X'),
    \end{aligned}
\end{equation}
where the last isomorphism follows from the Yoneda reduction lemma~\cite[Lm~1.2.2]{riley2018categories}.

Even though this category is equivalent to the definition of dependent lenses via the Grothendieck construction in~\cite{spivak2019generalized}, we believe it can have independent practical value. Encoding dependent lenses as maps
\begin{equation*}
    X\rightarrow M \times_B Y
    \quad \text{ and } \quad
    M \times_B Y'\rightarrow X' 
\end{equation*}
can lead to a more efficient implementation of, for instance, reverse-mode automatic differentiation (as done in the {\em Julia} library $\mathrm{Diffractor}$~\cite{fischer}), where the representative $M$ is optimized to contain precisely the information about the input that is required to compute the backward map. Features of the input that are not needed can be discarded, and quantities computed in the forward map can be stored in $M$ if they are useful for the backward map.

The documentation of the $\mathrm{Diffractor}$ library~\cite{fischer} hints at the need for dependent optics. Indeed, a key motivation for this work was to build a rigorous dependently-typed framework for bidirectional data transformations that would allow for reverse-mode automatic differentiation with an explicit notion of representative.
However, to formalize the difference between our construction of dependent lenses and the one based on functor lenses, we would need to define the {\em bicategory} of dependent optics, where all the information about the representative is preserved (cp. \cref{rm:bicategory}).

Unlike lenses, categories of dependent lenses admit finite coproducts, provided that the base category is {\em lextensive}~\cite[Sect.~4.4]{carboni1993introduction}.

\begin{lemma}\label{lm:coproducts_spans}
    If $\Cat$ is a lextensive category, then the inclusion $\Cat \hookrightarrow \Span_{\Cat}$ preserves coproducts.
\end{lemma}
\begin{proof}
    Let $A = \coprod_{i \in I} A_i$ be a coproduct in $\Cat$. Then, for all $B\in\ob(\Cat)$,
    \begin{equation*}
        \Span_{\Cat}\left(A, B\right) 
        = \Cat / (A \times B)
        \simeq\Cat / \coprod_{i \in I} \left(A_i \times B\right) 
        \simeq \prod_{i \in I} \Cat / \left(A_i \times B\right)
        = \prod_{i \in I} \Span_{\Cat}\left(A_i, B\right),
    \end{equation*}
    therefore $A$ is the coproduct of $(A_i)_{i \in I}$ in $\Span_{\Cat}$.
\end{proof}

\begin{proposition}\label{prop:coproducts_lenses}
    If $\Cat$ is a lextensive category, then $\DLens_{\Cat}$ has finite coproducts.
\end{proposition}
\begin{proof}
    By \cref{lm:coproducts_spans}, $\Span_{\Cat}$ has finite coproducts, given by coproducts in $\Cat$. It is straightforward to show that $\Cat/\anon$ turns coproducts into products, as
    \begin{equation*}
        \Cat / A = \Cat / \coprod_{i \in I} A_i \simeq \prod_{i \in I} \Cat / A_i. 
    \end{equation*}
    Thanks to \cref{prop:coproducts_optics}, $\DLens_{\Cat}$ has finite coproducts.
\end{proof}

\subsection{Dependent monoidal lenses}

The construction in \cref{sec:dependent_lenses} can be generalized to a symmetric monoidal category $(\Cat, \otimes)$ with reflexive equalizers that are preserved by the tensor product. This is analogous to the approach taken in~\cite{spivak2019generalized} to generalize lenses to symmetric monoidal categories via commutative comonoids.

Let $B$ be a comonoidal object in $\Cat$. Given a right $B$-comodule $M$ and a left $B$-comodule $N$, we can define their tensor product over $B$ as the following equalizer:
\begin{equation*}
    M \otimes_B N := \mathrm{eq}(M \otimes N \rightrightarrows M \otimes B \otimes N).
\end{equation*}
The category of commutative comonoids in $\ob(\Cat)$, denoted $\CComon_{\Cat, \otimes}$, has finite limits: the pullback of two $B$-coalgebras $Y_1, Y_2$ is isomorphic to the tensor product $Y_1 \otimes_B Y_2$ (see~\cite[C1.1~Lm.~1.1.8]{johnstone2002sketches} and subsequent discussion for the dual statement). Let $\CCoalg_{(\anon)}$ and $\Comod_{(\anon)}$ be the $\Span_{\CComon_{\Cat, \otimes}}$-indexed categories of commutative coalgebras and comodules respectively, where functoriality is given by extension and restriction of scalars. We define the category of {\em dependent monoidal lenses} as follows:
\begin{equation*}
    \DLens_{\Cat, \otimes} := \Optic_{\CCoalg_{(\anon)}, \Comod_{(\anon)}}.
\end{equation*}
A computation analogous to the one in \cref{eq:dependent_lens_derivation} yields the following explicit formula:
\begin{equation*}
    \DLens_{\Cat, \otimes}\left((X, X')^{A}, (Y, Y')^{B}\right) \simeq \coprod_{X \rightarrow Y} \Comod_A(X\otimes_B Y', X'),
\end{equation*}
where the morphism $X \rightarrow Y$ varies among comonoid homomorphisms.

\Cref{prop:coproducts_lenses} can be generalized to the monoidal case, establishing sufficient conditions for the existence of finite coproducts in $\DLens_{\Cat, \otimes}$. In the following proposition, we rely on the fact that, whenever $\Cat$ has finite coproducts, the forgetful functor $\CMon_{\Cat, \otimes} \rightarrow \Cat$ creates coproducts in $\CMon_{\Cat, \otimes}$. See the proof of~\cite[Prop.~1.2.14]{marty2009ouverts} for the dual statement, which concerns limits of commutative monoids rather than colimits of commutative comonoids.

\begin{proposition}\label{prop:coproducts_lenses_monoidal}
    Let $(\Cat, \otimes)$ be a symmetric monoidal category, with reflexive equalizers that are preserved by the tensor product. Let us assume that $\Cat$ has finite coproducts, and that for all finite coproduct of commutative comonoids $A = \coprod_{i \in I} A_i$, the map
    \begin{equation}\label{eq:equivalence_comodules}
        \Comod_{A} \rightarrow \prod_{i \in I}\Comod_{A_i},
        \text{ given by }
        M \mapsto (M\otimes_A A_i)_{i \in I},
    \end{equation}
    is an equivalence of categories. Then, $\DLens_{\Cat}$ has finite coproducts.
\end{proposition}
\begin{proof}
    The equivalence in \cref{eq:equivalence_comodules} is monoidal, hence the functor
    \begin{equation}\label{eq:equivalence_coalgebras}
        \CCoalg_{A} \rightarrow \prod_{i \in I}\CCoalg_{A_i},
        \text{ given by }
        X \mapsto (X\otimes_A A_i)_{i \in I},
    \end{equation}
    is an equivalence. As $\CCoalg_{(\anon)} = \CComon_{\Cat, \otimes}/\anon$, the category $\CComon_{\Cat, \otimes}$ is lextensive. Indeed, in the presence of pullbacks along coproduct injections, \cref{eq:equivalence_coalgebras} is a condition equivalent to extensivity, as shown in~\cite[Prop.~1.3]{lack2005adhesive}. Thanks to \cref{lm:coproducts_spans}, the category $\Span_{\CComon_{\Cat, \otimes}}$ has finite coproducts, given by coproducts in $\Cat$. It follows from \cref{eq:equivalence_comodules,eq:equivalence_coalgebras} that the functors $\Comod_{(\anon)}$ and $\CCoalg_{(\anon)}$ turn finite coproducts into products. Thanks to \cref{prop:coproducts_optics}, $\DLens_{\Cat, \otimes}$ has finite coproducts.
\end{proof}

\subsection{Dependent (monoidal) prisms}

Dependent prisms are dual to dependent lenses.
Let $\Cat$ be a finitely cocomplete category. Let $\Cospan_{\Cat} = \Span_{\Cat\superscript{{op}}}$ be its bicategory of cospans. Let $\anon / \Cat$ be the $\Cospan_{\Cat}$-indexed category of coslices. We define the category of {\em dependent prisms} as follows:
\begin{equation*}
    \DPrism_{\Cat} := \Optic_{\anon / \Cat, \anon / \Cat}.
\end{equation*}
Objects are given by spans $X \leftarrow A \rightarrow X'$.
Morphisms between two spans $X \leftarrow A \rightarrow X'$ and $Y \leftarrow B \rightarrow Y'$ are given by
\begin{equation}\label{eq:morphisms_prisms}
    \int^{M \in (A \sqcup B) / \Cat} A / \Cat(X, M \sqcup_B Y) \times A / \Cat(M \sqcup_B Y', X').
\end{equation}
As is the case for dependent lenses, the coend in \cref{eq:morphisms_prisms} can be computed explicitly:
\begin{equation*}
    \DPrism_{\Cat}\left((X, X')^{A}, (Y, Y')^{B}\right) \simeq \coprod_{Y' \rightarrow X'} {A / \Cat}(X, X' \sqcup_B Y).
\end{equation*}

Dependent monoidal prisms are dual to dependent monoidal lenses. Given a symmetric monoidal category $(\Cat, \otimes)$ with reflexive coequalizers that are preserved by the tensor product, let $\CMon_{\Cat, \otimes}$ be the category of commutative monoids in $\ob(\Cat)$. Let $\CAlg_{(\anon)}$ and $\Mod_{(\anon)}$ be the $\Cospan_{\CMon_{\Cat, \otimes}}$-indexed categories of commutative algebras and modules respectively, where functoriality is given by extension and restriction of scalars. We define the category of {\em dependent monoidal prisms} as follows:
\begin{equation*}
    \DPrism_{\Cat, \otimes} := \Optic_{\Mod_{(\anon)}, \CAlg_{(\anon)}}.
\end{equation*}
Morphisms can be computed via the explicit formula
\begin{equation*}
    \DPrism_{\Cat, \otimes}\left((X, X')^{A}, (Y, Y')^{B}\right) \simeq \coprod_{Y' \rightarrow X'} {A / \Cat}(X, X' \otimes_B Y),
\end{equation*}
where the morphism $Y' \rightarrow X'$ varies among monoid homomorphisms.

\subsection{Closed dependent optics}

Using a technique analogous to {\em coalgebraic optics}~\cite{riley2018categories}, it is sometimes possible to explicitly compute the coend in the $\Optic$ category using a right adjoint technique.

Let $\B$ be a bicategory, and $\icatL, \icatR$ be $\B$-indexed categories.
We say that $\Optic_{\icatL, \icatR}$ is a category of {\em closed dependent optics} if, for any $A, B \in \ob(\B)$ and $Y' \in \ob(\icatL(B))$, the functor $(\anon)\starL Y' \colon \B(A, B) \rightarrow \icatR^{A}$ has a right adjoint $Y' \triangleright \anon \colon \icatR^{A} \rightarrow \B(A, B)$.
Whenever that is the case, \cref{eq:morphisms} can be greatly simplified.
\begin{align*}
    \int^{f \in \B(A, B)} \icatL^{A}(X, f\starL Y) \times \icatR^{A}(f\starR Y', X') 
    &\simeq \int^{f \in \B(A, B)} \icatL^{A}(X, f\starL Y) \times \B(f, Y' \triangleright X')\\
    &\simeq \icatL^{A}(X, (Y' \triangleright X')\starL Y).
\end{align*}

A possible application of this technique is based on the bicategory of bimodules~\cite[Ex.~2.5]{benabou1967introduction} $\mathbf{Bimod}$ and on the $\mathbf{Bimod}$-indexed category $\Mod_{(\anon)}$. There, $Y' \triangleright X' = [Y', X']$, considered as an $(A, B)$-bimodule, hence morphisms in $\Optic_{\Mod_{(\anon)}, \Mod_{(\anon)}}$ can be computed explicitly:
\begin{equation*}
    \Optic_{\Mod_{(\anon)}, \Mod_{(\anon)}}\left((X, X')^{A}, (Y, Y')^{B}\right) = \Mod_A(X, [Y', X'] \otimes_B Y).
\end{equation*}

\section{Tambara representations}
\label{sec:tambara}

Tambara modules~\cite{pastro2007doubles} can be useful to define an {\em interface} for optics that does not depend on a choice of representative. Here, we adapt the notion of {\em generalized Tambara module} from~\cite{clarke2020profunctor} to the dependent case, and we generalize it to an arbitrary target category. For simplicity of notation, throughout this section we fix a bicategory $\B$ and two $\B$-indexed categories $\icatL$ and $\icatR$ (with coherence isomorphisms $\theta, \theta'$ respectively), and we write $\Optic$ instead of $\Optic_{\icatL, \icatR}$.

\begin{definition}\label{def:tambara}
    Let $\D$ be a category. A {\em $\D$-valued Tambara representation} consists of
    \begin{itemize}
        \item a functor $P^A \colon \left(\icatL^{A}\right)\superscript{op}\times \icatR^{A}\rightarrow \D$, for each object $A$ in $\B$,
        \item a natural transformation $\zeta_f \colon P^B(\anonfirst, \anonsecond)\Rightarrow P^A(f\starL\anonfirst, f\starR\anonsecond)$, for each morphism $f\colon A \rightarrow B$ in $\B$,
    \end{itemize}
    where $\zeta_f$ is extranatural in $f$ and satisfies the equations
    \begin{equation*}
        P^A(\theta_A, \theta_{A}^{\prime{-1}})\circ\zeta_{\id_A} = \id_{P^A}
        \quad \text{ and } \quad
        P^A(\theta_{f,g}, \theta_{f,g}^{\prime{-1}})\circ\zeta_{g\circ f}
        = (\zeta_{f})_{g\starL(\anonfirst),g\starR(\anonsecond)}\circ\zeta_{g},
    \end{equation*}
    for all $A, B, C \in \ob(\B)$, $f\colon A \rightarrow B$, and $g \colon B \rightarrow C$.
\end{definition}

\begin{remark}
    As the target category is arbitrary, $P$ does not correspond to a module over an enriched category, hence we find the name {\em representation} more appropriate.
\end{remark}

In more graphical terms, when $\D = \Set$ and thus $P^A, P^B$ are profunctors, the relationship between $P^A$, $P^B$, and $\zeta_f$ can be visualized via the following 2-cell in $\PROF$.
\begin{equation*}
    \begin{tikzcd}
        \icatL^B \arrow[d, "f\starL"{left}] \arrow[r,"|" anchor=center, "P^B"{yshift=2.5pt}, ""{name=U, inner sep=5pt,below}] &
        \icatR^B \arrow[d, "f\starR"]\\
        \icatL^{A} \arrow[r, "|" anchor=center, "P^A"{below,yshift=-2.5pt}, ""{name=D, inner sep=5pt}] &
        \icatR^{A}
        \arrow[Rightarrow, from=U, to=D, "\zeta_f"]
    \end{tikzcd}
\end{equation*}
$\Set$-valued Tambara representations can therefore be thought of as {\em lax $\B$-indexed profunctors}.

\begin{definition}
    Morphisms between Tambara representations $(P, \zeta), (Q, \zeta')$ are natural transformations $\eta^A\colon P^A \Rightarrow Q^A$ satisfying
    \begin{equation}\label{eq:tambara_morphism}
        \eta^A_{f\starL(\anonfirst), f\starR(\anonsecond)} \circ \zeta_f = \zeta'_f\circ \eta^B.
    \end{equation}
    $\D$-valued Tambara representations and their morphisms form a category, which we denote $\Tamb_{\D}$.
\end{definition}

Functors from $\Optic\superscript{op}$ to an arbitrary category can be described explicitly, thanks to \cref{thm:tambara_encoding}. The rest of the section is devoted to proving that result, via some intermediate steps, and exploring its consequences.

\subsection{The universal Tambara representation}
\label{sec:universal_tambara}

We aim to establish that all Tambara representations can be expressed as a composition of a functor with a {\em universal Tambara representation} $\iota\superscript{op}$. Here, we will define $\iota$ and prove that $\iota\superscript{op}$ is a Tambara representation. In \cref{sec:tambara_encoding}, we will show universality. 
\begin{definition}\label{def:iota}
    Let $A \in \ob(\B)$. The functor $\iota^A\colon \icatL^A \times (\icatR^A)\superscript{op} \rightarrow \Optic$ is defined as follows. Given an object $(X, X')$, $\iota^A(X, X') := (X, X')^A$. Given a morphism
    \begin{equation*}
        (l, r) \colon (X_0, X_0') \rightarrow (X_1, X_1'),
    \end{equation*}
    where $l\colon X_0 \rightarrow X_1$ and $r\colon X_1' \rightarrow X_0'$, we define $\iota^A(l, r)$ as the optic
    \begin{equation*}
        \angles{(\theta_A)_{X_1} \circ l}{r\circ(\theta^{\prime-1}_A)_{X_1'}} \colon (X_0, X_0')^A \rightarrow (X_1, X_1')^A
    \end{equation*}
    with representative $\id_A$.
\end{definition}

The following lemmas will make it much easier to do computations with $\iota$ and will allow us to prove that $\iota^A$ is indeed a functor and that $\iota\superscript{op}$ is a Tambara representation.

\begin{lemma}\label{lm:iota_right}
    Let $(X_0, X_0')^A, (X_1, X_1')^A \in \ob(\Optic)$. Let $l_1\colon X_0 \rightarrow X_1$ and $r_1\colon X_1' \rightarrow X_0'$. Then, for all optic $\angles{l_2}{r_2}$ with domain $(X_1, X_1')^A$,
    \begin{equation}\label{eq:iota_right}
        \angles{l_2}{r_2}\circ\iota^A(l_1, r_1) = \angles{l_2\circ l_1}{r_1\circ r_2}.
    \end{equation}
\end{lemma}
\begin{proof}
    We use twice the fact that $\theta_A$ (resp. $\theta^{\prime-1}_A$) is a natural transformation $\id_{\icatL^A} \Rightarrow \id_A\starL$ (resp. $\id_A\starR \Rightarrow \id_{\icatR^A}$). Specifically,
    \begin{align*}
        \id_A\starL(l_2) \circ (\theta_A)_{X_1} \circ l_1 = (\theta_A)_{X_2} \circ l_2 \circ l_1 = \id_A\starL(l_2 \circ l_1) \circ (\theta_A)_{X_0}\\
        r_1 \circ (\theta^{\prime-1}_A)_{X_1'} \circ \id_A\starR(r_2) = r_1\circ r_2 \circ (\theta^{\prime-1}_A)_{X_2'} = (\theta^{\prime-1}_A)_{X_0'} \circ \id_A\starR(r_1\circ r_2).
    \end{align*}
    As a consequence,
    \begin{align*}
        \angles{l_2}{r_2}\circ\iota^A(l_1, r_1)
        &=\angles{l_2}{r_2}\circ\angles{(\theta_A)_{X_1} \circ l_1}{r_1\circ(\theta^{\prime-1}_A)}_{X_1'}\\
        &= \angles{l_2\circ l_1}{r_1\circ r_2}\circ\angles{(\theta_A)_{X_0}}{(\theta^{\prime-1}_A)_{X_0'}}\\
        &= \angles{l_2\circ l_1}{r_1\circ r_2}.
    \end{align*}
\end{proof}

\begin{lemma}\label{lm:iota_left}
    Let $(Y_1, Y_1')^B, (Y_2, Y_2')^B \in \ob(\Optic)$.  Let $l_2\colon Y_1 \rightarrow Y_2$ and $r_2\colon Y_2' \rightarrow Y_1'$. Then, for all optic $\angles{l_1}{r_1}$ with codomain $(Y_1, Y_1')^B$ and representative $f$,
    \begin{equation}\label{eq:iota_left}
        \iota^B(l_2, r_2) \circ \angles{l_1}{r_1} = \angles{f\starL(l_2) \circ l_1}{r_1\circ f\starR(r_2)}.
    \end{equation}
\end{lemma}
\begin{proof}
    As $f\starL$ and $f\starR$ are functors, \cref{eq:morphism_composition} implies that
    \begin{align*}
        \iota^B(l_2, r_2) \circ \angles{l_1}{r_1}
        &= \angles{(\theta_B)_{Y_2} \circ l_2}{r_2\circ(\theta^{\prime-1}_B)_{Y_2'}} \circ \angles{l_1}{r_1}\\
        &= \angles{(\theta_B)_{Y_2}}{(\theta^{\prime-1}_B)_{Y_2'}} \circ \angles{f\starL(l_2) \circ l_1}{r_1\circ f\starR(r_2)}\\
        &= \angles{f\starL(l_2) \circ l_1}{r_1\circ f\starR(r_2)}.
    \end{align*}
\end{proof}

\begin{proposition}\label{prop:iota_functor}
    For all $A \in \ob(\B)$, $\iota^A$ is a functor.
\end{proposition}
\begin{proof}
    We must verify that $\iota^A$ preserves identity and composition of morphisms. Preservation of identity is straightforward, as
    \begin{equation*}
        \iota^A(\id_X, \id_{X'}) = \angles{(\theta_A)_{X}}{(\theta^{\prime-1}_A)_{X'}} = \id_{(X, X')^A}.
    \end{equation*}
    Let us now consider morphisms
    \begin{equation*}
        X_0 \xrightarrow{l_1} X_1 \xrightarrow{l_2} X_2
        \quad \text{ and } \quad
        X_2' \xrightarrow{r_2} X_1' \xrightarrow{r_1} X_0'.
    \end{equation*}
    Then, using \cref{eq:iota_right},
    \begin{align*}
        \iota^A(l_2, r_2) \circ \iota^A(l_1, r_1)
        &=\angles{(\theta_A)_{X_2}\circ l_2}{r_2 \circ (\theta^{\prime-1}_A)_{X_2'}}
        \circ \iota^A(l_1, r_1)\\
        &=\angles{(\theta_A)_{X_2}\circ l_2 \circ l_1}{r_1\circ r_2 \circ (\theta^{\prime-1}_A)_{X_2'}}\\
        &= \iota^A(l_2 \circ l_1, r_2 \circ r_1).
    \end{align*}
\end{proof}

\begin{proposition}\label{prop:iota_tambara}
    $\iota\superscript{op}$ is a $\Optic\superscript{op}$-valued Tambara representation, whose associated natural transformations are $\angles{\id_{f\starL(\anonfirst)}}{\id_{f\starR(\anonsecond)}}\superscript{op}$ (with representative $f$).
\end{proposition}
\begin{proof}
    By inverting all morphisms in \cref{def:tambara}, we work with the category $\Optic$ rather than $\Optic\superscript{op}$. Let $f\colon A \rightarrow B$. First, we need to ensure that
    \begin{equation*}
        \angles{\id_{f\starL(\anonfirst)}}{\id_{f\starR(\anonsecond)}}\colon
        \iota^A(f\starL(\anonfirst), f\starR(\anonsecond)) \Rightarrow \iota^B(\anonfirst, \anonsecond)
    \end{equation*}
    is a natural transformation. Let us consider morphisms $l \colon Y_0\rightarrow Y_1$ and $r\colon Y_1'\rightarrow Y_0'$.
    Then, by \cref{eq:iota_left,eq:iota_right},
    \begin{equation*}
        \iota^B(l, r)\circ \angles{\id_{f\starL Y_0}}{\id_{f\starR Y_0'}}
        = \angles{f\starL(l)}{f\starR(r)}
        = \angles{\id_{f\starL Y_1}}{\id_{f\starR Y_1'}}\circ\iota^A(f\starL(l), f\starR(r)).
    \end{equation*}
    Next, we must show that $\zeta_f$ is extranatural in $f$. Let us consider $f,g \colon A \rightrightarrows B$ and $m\colon f \Rightarrow g$. The following diagram commutes for all $Y \in \ob(\icatL^B)$, $Y' \in \ob(\icatR^B)$.
    \begin{equation*}
        \begin{tikzcd}
            \iota^A(f\starL Y, g\starR Y') \arrow[Rightarrow]{d} \arrow[Rightarrow]{r} &
            \iota^A(f\starL Y, f\starR Y') \arrow[Rightarrow]{d}\\
            \iota^A(g\starL Y, g\starR Y') \arrow[Rightarrow]{r} &
            \iota^B(Y, Y')
        \end{tikzcd}
    \end{equation*}
    This can be show by direct computation, using \cref{eq:iota_right}:
    \begin{align*}
        \angles{\id_{g\starL Y}}{\id_{g\starR Y'}}\circ\iota^A(\icatL(m)_{Y}, \id_{g\starR Y'})
        &= \angles{\icatL(m)_{Y}}{\id_{g\starR Y'}}\\
        &= \angles{\id_{f\starL Y}}{\icatR(m)_{Y'}}\\
        &= \angles{\id_{f\starL Y}}{\id_{f\starR Y'}}\circ\iota^A(\id_{f\starL Y}, \icatR(m)_{Y'}).
    \end{align*}
    Finally, we need to show the coherence laws for Tambara representations. The identity law is straightforward, as
    \begin{equation*}
        \angles{\id_{\id_A\starL X}}{\id_{\id_A\starR X'}}\circ\iota^A((\theta_A)_X, (\theta^{\prime-1}_A)_{X'}) = \angles{(\theta_A)_X}{(\theta^{\prime-1}_A)_{X'}} = \id_{(X, X')^A}.
    \end{equation*}
    Thanks to \cref{eq:iota_right}
    \begin{equation*}
        \angles{\id_{(g\circ f)\starL Z}}{\id_{(g\circ f)\starR Z'}}
        \circ \iota^A((\theta_{f, g})_{Z}, (\theta^{\prime-1}_{f, g})_{Z'})
        = \angles{(\theta_{f, g})_{Z}}{(\theta^{\prime-1}_{f, g})_{Z'}}.
    \end{equation*}
    Thanks to \cref{eq:morphism_composition},
    \begin{equation*}
        \angles{\id_{g\starL Z}}{\id_{g\starR Z'}}\circ \angles{\id_{f\starL(g\starL Z)}}{\id_{f\starR(g\starR Z')}}
        = \angles{(\theta_{f, g})_{Z}}{(\theta^{\prime-1}_{f, g})_{Z'}}.
    \end{equation*}
    Hence, the composition law holds and $\iota$ is a $\Optic\superscript{op}$-valued Tambara representation.
\end{proof}

\subsection{Tambara encoding}
\label{sec:tambara_encoding}

In this section, we establish that there is a functor $[\Optic\superscript{op}, \D] \rightarrow \Tamb_{\D}$ given by composition with the universal Tambara representation $\iota$. Furthermore, this functor is an isomorphism of categories, which we will show in \cref{thm:tambara_encoding}. This isomorphism will allow us to recover a classical {\em end formula} linking Tambara representations and optics. We start by showing that composition of a functor and a Tambara representation yields a Tambara representation.

\begin{proposition}\label{prop:functorial_tambara}
    Let $\Cat, \D$ be arbitrary categories. There is a functor $[\Cat, \D] \times \Tamb_{\Cat} \rightarrow \Tamb_{\D}$ given by composition.
\end{proposition}
\begin{proof}
    All conditions for Tambara representations are preserved by a functorial transformation. Given natural transformation $\mu\colon F\Rightarrow G$ and a morphism of Tambara representations $\eta\colon P \Rightarrow Q$, it is straightforward to verify that the horizontal composition of $\mu$ and $\eta$ is a morphism of Tambara representations $F \circ P \Rightarrow G \circ Q$:
    \begin{align*}
        \mu_{Q^A(X, X')}\circ F(\eta^A_{X, X'}) \circ F((\zeta_f)_{Y, Y'})
        &= \mu_{Q^A(X, X')}\circ F(\eta^A_{X, X'} \circ (\zeta_f)_{Y, Y'})\\
        &= \mu_{Q^A(X, X')}\circ F((\zeta'_f)_{Y, Y'}\circ \eta^B_{Y, Y'})\\
        &= G((\zeta'_f)_{Y, Y'}\circ \eta^B_{Y, Y'})\circ \mu_{P^B(Y, Y')}\\
        &= G((\zeta'_f)_{Y, Y'})\circ G(\eta^B_{Y, Y'})\circ \mu_{P^B(Y, Y')}\\
        &= G((\zeta'_f)_{Y, Y'})\circ \mu_{Q^B(Y, Y')} \circ F(\eta^B_{Y, Y'}).
    \end{align*}
\end{proof}

\begin{theorem}\label{thm:tambara_encoding}
    Let $\D$ be a category. Let $\iota\superscript{op}$ be the universal Tambara representation. Then the functor 
    \begin{equation*}
        \anon \circ \iota\superscript{op}\colon [\Optic\superscript{op}, \D] \rightarrow \Tamb_{\D}
    \end{equation*}
    is an isomorphism of categories.
\end{theorem}
\begin{proof}
    Let $P$ be a $\D$-valued Tambara representation, with associated natural transformations $\zeta_f$. Let us define $\tilde P\colon \Optic\superscript{op} \rightarrow \D$ as follows. On objects,
    \begin{equation*}
        \tilde P\left((X, X')^{A}\right) = P^A(X, X').
    \end{equation*}
    To extend $\tilde P$ to morphisms, we proceed as follows. As $\zeta_f$ is extranatural in $f$, the map
    \begin{align*}
        \icatL^A(X, f\starL Y) \times \icatR^A(f\starR Y', X')
        &\rightarrow \D\left(
            P\left((Y, Y')^{B}\right),
            P\left((X, X')^{A}\right)
        \right)\\
        \angles{l}{r} &\mapsto P^A(l, r)\circ(\zeta_f)_{Y, Y'}
    \end{align*}
    induces a map
    \begin{equation*}
        \tilde P\colon\Optic\left((X, X')^{A}, (Y, Y')^{B}\right) \rightarrow \D\left(
            P\left((Y, Y')^{B}\right),
            P\left((X, X')^{A}\right)
        \right).
    \end{equation*}
    Preservation of identity and composition follows from the coherence laws of \cref{def:tambara}, hence $\tilde P$ is a functor.

    It is straightforward to verify that $\tilde P \circ \iota\superscript{op} = P$ as Tambara representations. On objects,
    \begin{equation*}
        (\tilde{P}\circ\iota\superscript{op})^A(X, X') = \tilde{P}((X, X')^A) = P^A(X, X').
    \end{equation*}
    On morphisms, given $l\colon X_0\rightarrow X_1$ and $r\colon X_1'\rightarrow X_0'$,
    \begin{align*}
        (\tilde{P}\circ\iota\superscript{op})^A(l, r)
        &= \tilde{P}(\angles{(\theta_A)_{X_1} \circ l}{r\circ(\theta^{\prime-1}_A)_{X_1'}})\\
        &= P((\theta_A)_{X_1} \circ l, r\circ(\theta^{\prime-1}_A)_{X_1'})\circ (\zeta_{\id_A})_{X_1, X_1'}\\
        &= P(l, r) \circ P((\theta_A)_{X_1}, (\theta^{\prime-1}_A)_{X_1'})\circ (\zeta_{\id_A})_{X_1, X_1'}\\
        &= P(l, r).
    \end{align*}
    Finally,
    \begin{equation*}
        \tilde P (\angles{\id_{f\starL Y}}{\id_{f\starR Y'}})
        = P^A(\id_{f\starL Y}, \id_{f\starR Y'}) \circ (\zeta_f)_{Y, Y'}
        = (\zeta_f)_{Y, Y'},
    \end{equation*}
    hence $\tilde P \circ \iota\superscript{op} = P$ as Tambara representations, so $\anon \circ \iota\superscript{op}$ is surjective on objects.
    
    Let $P, Q$ be Tambara representations, with associated natural transformation families $\zeta, \zeta'$ respectively. Let $\eta\colon P \Rightarrow Q$ be a morphism of Tambara representations. Then, we can define a natural transformation
    \begin{equation*}
        \tilde \eta \colon \tilde P \Rightarrow \tilde Q
        \quad \text{ given by } \quad
        \tilde \eta_{(X, X')^{A}} := \eta^A_{X, X'}.
    \end{equation*}
    Naturality can be verified via a direct computation:
    \begin{align*}
        Q(\angles{l}{r})\circ \tilde \eta_{(Y, Y')^{B}}
        &= Q^A(l, r)\circ (\zeta'_f)_{Y, Y'}\circ \eta^B_{Y, Y'}\\
        &= Q^A(l, r)\circ \eta^A_{f\starL Y, f\starR Y'} \circ (\zeta_f)_{Y, Y'}\\
        &= \eta^A_{X, X'} \circ P^A(l, r) \circ (\zeta_f)_{Y, Y'}\\
        &= \tilde \eta_{(X, X')^{A}} \circ P(\angles{l}{r}).
    \end{align*}
    Here, we have used \cref{eq:tambara_morphism} and the fact that $\eta^A$ is a natural transformation from $P^A$ to $Q^A$. As $\tilde \eta_{\iota^A(X, X')} = \tilde \eta_{(X, X')^A} = \eta^A_{X, X'}$, the functor $\anon \circ \iota\superscript{op}$ is full. Finally, if for all $A \in \ob(\B), X \in \ob(\icatL^A), X' \in \ob(\icatR^A)$,
    \begin{equation*}
        \eta^A_{X, X'} = \mu_{\iota^A(X, X')}
    \end{equation*}
    with $\mu \colon \tilde P \Rightarrow \tilde Q$, then $\mu = \tilde \eta$. Hence, the functor $\anon \circ \iota\superscript{op}$ is faithful.
\end{proof}

\Cref{thm:tambara_encoding} has wide practical applications, as it implies that a Tambara representation $P$ can be used as an {\em interface} for optics. In other words, an optic $o \in \Optic\left((X, X')^A, (Y, Y')^B\right)$ can be encoded as a family of morphisms
\begin{equation*}
    P^B(Y, Y') \rightarrow P^A(X, X'),
\end{equation*}
where $P$ varies among $\D$-valued Tambara representations. This encoding has several practical advantages. It does not depend on a choice of representative, it simplifies composition of optics---replacing the rule in \cref{eq:morphism_composition} with standard function composition---and it can be used to compose optics of different types, as discussed in~\cite{clarke2020profunctor}. Furthermore, the dependent version of a classical result---the {\em profunctor representation theorem}~\cite[Thm.~4.14]{clarke2020profunctor}---is a direct consequence of \cref{thm:tambara_encoding}, specialized to the case $\D = \Set$.

\begin{lemma}\label{lm:end_formula}
    Let $\Cat$ be a locally small category, and let $\hat\Cat$ denote its category of presheaves. Then, for all $S, T \in \ob(\Cat)$,
    \begin{equation}\label{eq:end_formula}
        \Cat(S, T) \simeq \int_{F \in \hat\Cat} \Set(F(T), F(S)).
    \end{equation}
\end{lemma}
\begin{proof}
    Let us consider the representable presheaf $\Cat(\anon, T)$. Applying the Yoneda reduction lemma~\cite[Lm.~1.2.2]{riley2018categories} to the evaluation-at-$S$ functor $\hat\Cat \rightarrow \Set$, we obtain
    \begin{equation*}
        \Cat(S, T) = \Cat(\anon, T)(S) \simeq \int_{F \in \hat\Cat} \Set(\hat\Cat(\Cat(\anon, T), F), F(S))
        \simeq \int_{F \in \hat\Cat} \Set(F(T), F(S)),
    \end{equation*}
    where in the last step we have used the Yoneda lemma to compute $\hat\Cat(\Cat(\anon, T), F)$.
\end{proof}

\begin{theorem}\label{thm:tambara_representation}\cite[Thm.~3.5]{capucci2022seeing}
    Let $(X, X')^A$ and $(Y, Y')^B$ be objects in $\Optic$. Then,
    \begin{equation}\label{eq:tambara_representation}
        \Optic\left((X, X')^A, (Y, Y')^B\right) \simeq
        \int_{P \in \Tamb_\Set} \Set\left(P^B(Y, Y'), P^A(X, X')\right).
    \end{equation}
\end{theorem}
\begin{proof}
    By~\cref{thm:tambara_encoding}, $\Tamb_\Set$ is isomorphic to the category of presheaves over $\Optic$. Hence, \cref{eq:tambara_representation} follows from \cref{eq:end_formula}, with $\Cat = \Optic$, $S = (X, X')^A$, and $T = (Y, Y')^B$.
\end{proof}

\section{Discussion}

In this work, we developed a theory of {\em dependent optics} that simultaneously generalizes (mixed) optics and functor lenses. Natural examples of this construction arise from finitely complete (or finitely cocomplete) categories or, more generally, from symmetric monoidal categories with reflexive equalizers (or reflexive coequalizers) preserved by the tensor product. Motivated by the practical applicability of coproducts of dependent lenses~\cite{braithwaite2021fibre,capucci2021translating,fischer}, we showed sufficient conditions under which the category of dependent optics admits finite coproducts. In~\cite{moeller2020monoidal,spivak2019generalized}, it was shown that the category of functor lenses admits a monoidal structure whenever the underlying indexed category is monoidal. Although we did not pursue this direction here, we believe that the key ingredient used in~\cite{moeller2020monoidal}---the notion of {\em pseudomonoid} in the 2-category of indexed categories---can be adapted to our setting to obtain an analogous result for the category of dependent optics.

Aiming to mimic the {\em profunctor encoding}~\cite{clarke2020profunctor,milewski2017} of optics, we defined a generalization of Tambara modules---{\em Tambara representations}. We showed that contravariant functors from $\Optic$ to an arbitrary category $\D$ can equivalently be described as $\D$-valued Tambara representations. Using this result, we established a {\em representative-free} interface for dependent optics, where each dependent optic is encoded as a polymorphic function, and recovered the {\em profunctor representation theorem}~\cite{clarke2020profunctor} in our setting. In the future, it will be interesting to explore in which particular cases of dependent optics this general result can be specialized to yield simplified encodings, akin to the {\em van Laarhoven encoding} for lenses~\cite{riley2018categories}.

A more general definition of optics, {\em fiber optics}, has been developed in~\cite{braithwaite2021fibre}. Furthermore, the authors sketched a possible formalization of the {\em bicategory} of dependent optics to simultaneously generalize dependent lenses, mixed optics, and fiber optics. We believe that novel avenues of research can arise from the interplay between the two works. From the construction in~\cite[Sect.~4.3]{braithwaite2021fibre}, it is straightforward to see that fiber optics are a particular case of dependent optics, as developed here, and they can probably be used to unify many examples of dependent optics. Hopefully, the work done here can help fine-tune the technical details of the definition of the bicategory of dependent optics.

Even though they appear different on the surface, our definition of dependent optics and the notion of {\em compound optics}~\cite{milewski2022compound} are equivalent (see~\cite[Ex.~4.2]{capucci2022seeing}). In our view, this has several beneficial consequences. On the one hand, our manuscript can be used to fill the gaps left in~\cite{milewski2022compound}, such as the study of the properties of the category of dependent optics or the generalization of Tambara modules and of the profunctor representation theorem. On the other hand, the approach taken in~\cite{milewski2022compound} offers a different perspective on dependent optics and their composition in terms of Kan extensions, which can help form an intuitive understanding of our direct, explicit definitions. Finally, the existence of two equivalent, independently-developed definitions supports the intuition that this is indeed a principled adaptation of optics to the dependent case.

\subsubsection*{Acknowledgements}

The author is grateful to Mattia G. Bergomi for helpful comments and feedback, and to Keno Fischer for interesting discussions on the interplay between the optics formalism and reverse-mode automatic differentiation. Thanks too to Bartosz Milewski for valuable feedback on the first version of the manuscript.

\bibliographystyle{eptcs}
\bibliography{bibliography}

\end{document}